\documentclass[11pt]{amsart}

\usepackage{texmaX,semmaX,semtkzX}
\usepackage{makecell}
\usepackage{multirow}

\usepackage{booktabs} 
\usepackage{mathtools}
\usepackage{amssymb}
\usepackage{listings}

\usepackage{verbatim}  
\usepackage{enumitem}

\usepackage{hyperref}

\usepackage{float}

\usepackage{lmodern}
\usepackage{minted} 

\usepackage{listings}
\usepackage[T1]{fontenc}

\usepackage[table]{xcolor}

\usepackage[margin=1in]{geometry}

\newtheorem*{unnumberedquestion}{Question}

\usepackage{amsthm}

\DeclareFontFamily{U}{wncy}{}
\DeclareFontShape{U}{wncy}{m}{n}{<->wncyr10}{}
\DeclareSymbolFont{mcy}{U}{wncy}{m}{n}
\DeclareMathSymbol{\Sha}{\mathord}{mcy}{"58}

\DeclareMathOperator{\Ker}{Ker}

\usepackage{tikz}
\usetikzlibrary{arrows.meta}
\usetikzlibrary{arrows.meta}
\usetikzlibrary{positioning, arrows.meta}

\usepackage[utf8]{inputenc}

\usepackage{placeins}

\usepackage{float}
\usepackage{caption}

\usepackage{xcolor}
\usepackage{ulem}

\begin{document}

\title{When the BSD Real Period Remembers Elliptic Curves over $\mathbb{Q}$}

\author{Asuka Shiga}
\address{Mathematical Institute, Graduate School of Science, Tohoku University, 6-3 Aramaki Aza Aoba, Aoba-ku, Sendai, Miyagi 980-8578, Japan.}
\email{shiga.asuka.r5@dc.tohoku.ac.jp}

\begin{abstract}
We study the extent to which the invariants appearing in the Birch and Swinnerton-Dyer conjecture remember the elliptic curve, and we give several affirmative answers to this question. For example, the real period $\Omega_E
=\frac{\Gamma(1/4)^2}{2\sqrt{\pi}\,17^{1/4}}$ and the Tate--Shafarevich group $
\Sha(E/\mathbb{Q})\cong \mathbb{Z}/2\mathbb{Z}\times \mathbb{Z}/2\mathbb{Z}$ remember the elliptic curve
$E: y^2=x^3+17x$.

We also prove that the number of isomorphism classes of elliptic curves
over $\mathbb{Q}$ having the same BSD real period is at most $6$. Our result can be viewed as an analogue of Kenku's theorem for real
periods: while Kenku's theorem states that a $\mathbb{Q}$-isogeny class
contains at most eight elliptic curves up to isomorphism, at most six
elliptic curves over $\mathbb{Q}$, up to isomorphism, can have the same
BSD real period. In this sense, the BSD real period has greater
reconstruction power than the $L$-function.

We then prove that an elliptic curve $E$ over $\mathbb{Q}$ is determined up to isomorphism over $\mathbb{Q}$ by its real period, the number of connected components of $E(\mathbb{R})$, and the real period of the quadratic twist $E^{D}$ by a negative square-free integer $D\equiv 1 \bmod 4$ coprime to the minimal discriminant of $E$. Moreover, over $\mathbb{Q}$, both statements already hold when the real periods are known only to sufficiently many decimal places, where the required precision depends on $E$ (and $D$).
\end{abstract}

\keywords{Real period, isogeny}

\maketitle

\setcounter{tocdepth}{2}
\tableofcontents

\section{Introduction}

We ask to what extent the invariants appearing in the Birch and Swinnerton-Dyer conjecture remember the elliptic curve itself.

\begin{unnumberedquestion}
Can an elliptic curve be recovered up to isomorphism over $\mathbb{Q}$ from the invariants appearing in the BSD conjecture?
\end{unnumberedquestion}

For example, what is the elliptic curve over $\Bbb{Q}$ which has the data \[\Omega_E
=
\frac{\Gamma(1/4)^2}{2\sqrt{\pi}\,17^{1/4}},
\Sha(E/\mathbb{Q})\cong \mathbb{Z}/2\mathbb{Z}\times \mathbb{Z}/2\mathbb{Z} \ ?\] 

Here, $\Omega_E$ denotes the real period of $E$, $\Sha(E/\mathbb{Q})$ denotes the Tate--Shafarevich group of $E/\mathbb{Q}$, and $\Gamma$ denotes the gamma function.

The answer is unique and is given in Example \ref{prop:j1728}.

In contrast, one can find infinitely many pairs of non-isomorphic quadratic twists sharing all BSD invariants; we call such pairs “BSD twins” \cite{shiga}.

The aim of this paper is to present an affirmative aspect of this question.

The equality of $L$ functions is equivalent, by Faltings's theorem, to the existence of an isogeny between the curves, and hence imposes a strong constraint. On the other hand, some of the individual invariants appearing in the BSD formula, such as the Tate--Shafarevich group, by themselves carry much less information about the isomorphism class of the curve. We focus on the real period and prove that it has the following reconstruction property: it determines the $\mathbb{Q}$-isogeny class, and, together with the number of real connected components and the real period of a suitable negative quadratic twist, it determines the $\mathbb{Q}$-isomorphism class of the elliptic curve.

First, as a consequence of Wüstholz’s analytic subgroup theorem, we prove the following. This theorem is a corollary of Theorem \ref{periodmain}, the number-field version proved in Section 2.

\begin{theorem}\label{ourmain}
Let \(E_1\) and \(E_2\) be elliptic curves over \(\mathbb{Q}\).
\begin{enumerate}

\item If their real periods satisfy
\[
\frac{\Omega_{E_1}}{\Omega_{E_2}} \in \mathbb{Q}^{\times},
\]
then \(E_1\) and \(E_2\) are \(\mathbb{Q}\)-isogenous.

\item Let $a \in \mathbb{Q}_{>0}$ and let $E_1/\mathbb{Q}$ be an elliptic curve.
Then there exists $\varepsilon(E_1,a)>0$ such that, for every elliptic curve
$E_2/\mathbb{Q}$, if
\[
\left|
\frac{\Omega_{E_1}}{\Omega_{E_2}}-a
\right|
<\varepsilon(E_1,a),
\]
then $E_1$ and $E_2$ are isogenous over $\mathbb{Q}$.

\item (=Theorem \ref{atmost6}) The number of isomorphism classes of elliptic curves
over $\mathbb{Q}$ having the same real period
is at most $6$.

\end{enumerate}
\end{theorem}

A version of part (1) asserting isogeny over $\overline{\mathbb Q}$ appears in Exercise~8.8 of \cite{lecture} and Corollary~1 of \cite{mova}, although these references do not explicitly state that the BSD real period can be used. Although part (1) may be regarded as folklore, we provide a detailed proof, including the descent of the isogeny to $\mathbb Q$ and the justification that the BSD real period is indeed a period in the sense that it can be expressed as the integral of an invariant differential over a homology cycle. These details are needed for the reconstruction problem considered here. 
By Theorem~\cite{Kenku}, the isogeny graph has at most $8$ vertices. In particular, there are at most $8$ elliptic curves over $\mathbb Q$, up to $\mathbb Q$-isomorphism, sharing the same $L$-function. However, part (3) of Theorem \ref{ourmain} shows that the real period has stronger reconstruction power than the $L$-function. As a striking manifestation of this stronger reconstruction property, there are at most $6$ elliptic curves over $\mathbb Q$, up to $\mathbb Q$-isomorphism, sharing the same real period.

\begin{corollary}
Assume the Birch and Swinnerton-Dyer conjecture. If $\mathrm{Reg}(E_1/\Bbb{Q})=\mathrm{Reg}(E_2/\Bbb{Q})$, and the leading terms of the $L$-functions of $E_1/\Bbb{Q}$ and $E_2/\Bbb{Q}$ coincide, then $E_1$ and $E_2$ are isogenous over $\Bbb{Q}$.
\end{corollary}

\begin{proof}
Since the leading terms of the $L$-functions agree, the ratio of the real periods is rational. Hence $E_1$ and $E_2$ are isogenous over $\Bbb{Q}$.\end{proof}

Thanks to this theorem, we can characterize elliptic curves over $\mathbb{Q}$ in terms of their real periods and other BSD invariants. For example, the following result shows that the elliptic curve $E:y^2=x^3+17x$
is determined solely by its real period and Tate--Shafarevich group.
The curve $E$ is the Jacobian of the classical Lind--Reichardt curve $2y^2=x^4-17$, which gives the first known counterexample to the Hasse principle.

\begin{example}\label{prop:j1728}
There is a unique $\mathbb{Q}$-isomorphism class of elliptic curves
$E/\mathbb{Q}$ whose real period and Tate--Shafarevich group are given by
\[
\Omega_E
=
\frac{\Gamma(1/4)^2}{2\sqrt{\pi}\,17^{1/4}},
\Sha(E/\mathbb{Q})
\cong
\mathbb{Z}/2\mathbb{Z}\times \mathbb{Z}/2\mathbb{Z}.
\]

A direct computation shows that \(E: y^2=x^3+17x\) has the given data $\Omega_E, \Sha(E/\Bbb{Q})$. Indeed, we have $\Omega_E
=
\frac{\Gamma(1/4)^2}{2\sqrt{\pi}\,17^{1/4}}, \Sha(E/\mathbb{Q})
\cong
\mathbb{Z}/2\mathbb{Z}\times \mathbb{Z}/2\mathbb{Z}$. For the computation of the Tate--Shafarevich group, see \cite{Rubin} for the odd part. For the even part, a 2-descent shows $\Sha(E/\Bbb{Q})[2]=(\Bbb{Z}/2\Bbb{Z})^2$ and $\Sha(E/\Bbb{Q}(\sqrt{-2}))[2]=0$. Then, by Proposition 5.6 (2) of \cite{behavior}, we have $\Sha(E/\Bbb{Q})[2^{\infty}]=(\Bbb{Z}/2\Bbb{Z})^2$.

Let \(E'\) be another elliptic curve over \(\Bbb{Q}\). Then \(E\) and \(E'\) are isogenous over \(\Bbb{Q}\) since they share the same real period. Then \(E'\) is \(2\)-isogenous to \(E\) over \(\Bbb{Q}\) (by the isogeny graph of LMFDB \cite{lmfdb}).

So \(E'\) is isomorphic to \(y^2=x^3-68x\). However, a 2-descent computation in SageMath shows that $\Sha(E'/\mathbb{Q})[2]=0$. Thus, the curve is unique, namely \(E\). \hfill $\lozenge$

\end{example}

Second, we prove the following theorem. For an elliptic curve $E/\mathbb Q$ and a square-free integer $D$, let $E^D$ denote the quadratic twist of $E$ by $D$, and let $\Omega_E$ denote the real period of $E$.

The following gives a sufficient condition under which another isogenous candidate $E'$ is in fact $\mathbb{Q}$-isomorphic to $E$. In other words, it shows that $E/\mathbb{Q}$ can be recovered solely from data intrinsic to $E$.

\begin{theorem}[=Theorem \ref{recover}]
Let $E/\mathbb{Q}$ be an elliptic curve, and let $D<0$ 
be a square-free integer satisfying
\[
D\equiv 1\bmod 4,\quad \gcd(D,\Delta_E)=1,
\]
where $\Delta_E$ denotes the minimal discriminant of $E$.
Then the following statements hold.

\begin{enumerate}
\item[(1)]
For any elliptic curve $E'/\mathbb{Q}$, if
\[
c_{E/\mathbb{R}}=c_{E'/\mathbb{R}},\quad
\Omega_E=\Omega_{E'},\quad
\Omega_{E^D}=\Omega_{E'^D},
\]
then $E$ and $E'$ are isomorphic over $\mathbb{Q}$.

\item[(2)]
There exists an integer $N=N(E,D)\geq 1$ such that,
for any elliptic curve $E'/\mathbb{Q}$, if
\[
c_{E/\mathbb{R}}=c_{E'/\mathbb{R}},\quad
R_N(\Omega_E)=R_N(\Omega_{E'}),\quad
R_N(\Omega_{E^D})=R_N(\Omega_{E'^D}),
\]
then $E$ and $E'$ are isomorphic over $\mathbb{Q}$.
\end{enumerate}

Here, $c_{E/\mathbb{R}}$ denotes the number of
connected components of $E(\mathbb{R})$, and
\[
R_N(x):=10^{-N}\left\lfloor 10^Nx+\frac12\right\rfloor
\quad (x>0).
\]
\end{theorem}

Even if non-isomorphic elliptic curves $E$ and $E'$ become BSD twins for infinitely many values of $D$, their BSD real periods cannot remain identical for every $D$. For some quadratic twist, either the real periods or the numbers of real connected components must differ, thereby detecting the non-isomorphism. In other words, BSD twins can be distinguished by passing to a suitable quadratic twist.

\subsection{Notation}

\begin{itemize}

\item $\Omega_{E}$: The real period of $E$, given by $\Omega_E = \int_{E(\mathbb{R})} |\omega|$ where $\omega$ is the invariant differential of a global minimal model of $E/\Bbb{Q}$. 

\item For abelian varieties \(A\) and \(B\) over \(K\), the notation
\[
A\sim_K B
\]
means that \(A\) and \(B\) are isogenous over \(K\).

\item $E^D/\mathbb{Q}$: For the Weierstrass model $E/\mathbb{Q}: y^2=x^3+ax^2+bx+c$, the quadratic twist of $E/\mathbb{Q}$ by a square-free integer $D$ is an elliptic curve defined by the equation \[E^D: y^2=x^3+aDx^2+bD^2x+cD^3.\]    

\item We put
\[
c_{E/\mathbb{R}}:=[E(\mathbb{R}):E^0(\mathbb{R})]\in\{1,2\},
\]
where $E^0(\mathbb{R})$ denotes the identity component of $E(\mathbb{R})$ for an elliptic curve over $\Bbb{R}$.

\item \(A\ll B\) means that there exists a constant \(C>0\) such that \(|A|\leq CB\) for all values of the variables under consideration.
\end{itemize}

\section{Real period remembers the rational isogeny class }
Let \(E/\mathbb{Q}\) be an elliptic curve. Complex conjugation acts on \(H_1(E(\mathbb{C}),\mathbb{Z})\).  We define $H_1^+(E(\mathbb{C}),\mathbb{Z})
:=
\left\{
r\in H_1(E(\mathbb{C}),\mathbb{Z})
\mid \overline{r}=r
\right\}$, and $H_1^-(E(\mathbb{C}),\mathbb{Z})
:=
\left\{
r\in H_1(E(\mathbb{C}),\mathbb{Z})
\mid \overline{r}=-r
\right\}$.

Let \(\omega_{\min}\) be a Néron differential.  We put
\[
\Omega_E := \int_{E(\mathbb{R})} |\omega_{\min}|,
\]
and, for generators
\[
r^+\in H_1^+(E(\mathbb{C}),\mathbb{Z}), r^-\in H_1^-(E(\mathbb{C}),\mathbb{Z}),
\]
define
\[
\Omega_E^+
:=
\left|
\int_{r^+} \omega_{\min}
\right|, \Omega_E^-
:=
\int_{r^-} \omega_{\min}.
\]
Here \(\Omega_E^+\) is real and positive, while \(\Omega_E^-\) is purely imaginary.

\begin{lemma}\label{plusminus}
\[
\Omega_E = c_{E/\Bbb{R}}\Omega^{+}_E.
\]
\end{lemma}

\begin{proof}
If $E(\Bbb{R})$ is connected, then $\Omega_E=\Omega_{E}^+$. If $E(\Bbb{R})$ is disconnected, each connected component of $E(\mathbb R)$ is a translate of $E^0(\mathbb R)$, and the invariant differential $\omega_E$ is translation-invariant. Thus, $\Omega_E=2\Omega_{E}^+$. \end{proof}

\begin{proposition}\label{isogenousperiod}
Let \(E_1/\mathbb Q\) and \(E_2/\mathbb Q\) be elliptic curves
which are connected by an isogeny defined over \(\mathbb Q\).
Then
\[
\frac{\Omega_{E_2}}{\Omega_{E_1}}\in \mathbb Q.
\]
\end{proposition}

\begin{proof}
See Lemma 7.4 of \cite{dokchitser}. This is the case of $\mathcal{K}=\Bbb{R}$.\end{proof}

For an algebraic group \(G\) over a field \(k\), we write
\(\operatorname{Lie}(G)\) for the tangent space of \(G\) at the identity.
Equivalently,
\[
\operatorname{Lie}G=\Ker\bigl(G(k[\varepsilon])\to G(k)\bigr),
\varepsilon^2=0
\] where $G(k[\varepsilon])\to G(k)$ is a truncation map defined by $\epsilon \mapsto 0$. We define $\mathrm{exp}_G:\mathrm{Lie}(G)\to G(\mathbb{C})$ as the unique holomorphic group homomorphism whose differential at $0$ is the identity.

\begin{theorem}[Analytic subgroup theorem \cite{annals}]\label{annals}
Let $G$ be a connected commutative algebraic group defined over $\overline{\Bbb{Q}}$. 

Take \(u \in \operatorname{Lie}(G(\mathbb{C}))\) such that
\[
\exp_G(u) \in G(\overline{\mathbb{Q}}).
\]

Let $V_u$ be the smallest $\overline{\Bbb{Q}}$-vector subspace of $\mathrm{Lie}(G)$ such that $u$ lies in the complex vector space $V_u\otimes_{\overline{\Bbb{Q}}}\Bbb{C}$. Then there exists a connected algebraic subgroup $H_u$ of $G$ defined over $\overline{\Bbb{Q}}$ satisfying $\mathrm{Lie}(H_u)=V_u$.
\end{theorem}

\begin{lemma}
Let \(G\) be a connected algebraic group over
\(k=\overline{\mathbb{Q}}\).

\begin{enumerate}
\item For every algebraic subgroup \(H\subset G\), one has
\[
\dim_k \operatorname{Lie}(H)=\dim H.
\]

\item The map
\[
H \longmapsto \operatorname{Lie}(H)
\]
from connected algebraic subgroups \(H\) of \(G\) to Lie subalgebras
of \(\operatorname{Lie}(G)\) is injective and inclusion preserving.
\end{enumerate}
\end{lemma}
\begin{proof}
(1) See Proposition 1.37 of \cite{MilneAG}.
(2) See Proposition 10.15 and Corollary 10.16 of \cite{MilneAG}.\end{proof}

\begin{definition}
Let \(E/\mathbb C\) be an elliptic curve and let \(\omega\) be a nonzero
invariant differential on \(E\). Since \(\dim_{\mathbb C}\operatorname{Lie}(E)=1\)
and \(\omega\neq 0\), the map \(\omega:\operatorname{Lie}(E)\to\mathbb C\) is an
isomorphism; let \(e_{\omega}\in \operatorname{Lie}(E)\) be the unique tangent
vector dual to \(\omega\), namely
\[
\omega(e_{\omega})=1.
\]

We fix an isomorphism
\[
\iota: \mathbb C \xrightarrow{\sim} \operatorname{Lie}(E), z\longmapsto z e_{\omega},
\]
we define a holomorphic group homomorphism $\exp_{E,\omega}:\mathbb C\longrightarrow E(\mathbb C), z\longmapsto \exp_E(z e_{\omega})$.
Define
\[
\Lambda(E,\omega)
\stackrel{\text{def}}{=}
\Ker(\exp_{E,\omega}).
\]
\end{definition}

\begin{proposition}\label{delicate}
Let
\[
\Lambda:=\left\{
\int_{\gamma}\omega
\;\middle|\;
\gamma\in H_1(E(\mathbb C),\mathbb Z)
\right\}.
\]
Then, \(\Lambda(E,\omega)=\Lambda\).
\end{proposition}

\begin{proof}(\textit{cf.} Proposition 8.15 of \cite{lecture})
By Proposition VI.5.2 of \cite{sil} (The statement there is written
for the standard invariant differential; the general case follows by scaling,
since \(\Lambda(E,c\omega)=c\,\Lambda(E,\omega)\) and
\(\int_\gamma c\omega=c\int_\gamma\omega\) for \(c\in\mathbb C^{\times}\)), \(\Lambda\) is a lattice and
\[
E(\mathbb C)\xrightarrow{\ \cong\ }\mathbb C/\Lambda,
\qquad P\longmapsto \int_{O}^{P}\omega,
\]
is an isomorphism of complex Lie groups (see Proposition 5.2 (b) of \cite{sil}). 

Let \(F:\mathbb C/\Lambda\cong  E(\mathbb C)\) be the isomorphism map of uniformization. Let \(\pi:\mathbb C\to\mathbb C/\Lambda\) be the projection. 
We have $(F\circ \pi)^*\omega=dz$ by the proof of Proposition 5.2 of \cite{sil}.

Therefore, the pullback of $\omega$ by the two map $\Bbb{C}\to E(\Bbb{C})$ in the following diagram is the same when $\omega(e_{\omega})=1$.
{\footnotesize
\[
\begin{tikzcd}[column sep=large, row sep=large]
\mathbb{C}
  \arrow[r, "\iota:\cong"]
  \arrow[dr, "\pi"    ]
&
\operatorname{Lie} E
  \arrow[dr, "\exp_E"]
&
\\
&
\mathbb{C}/\Lambda
  \arrow[r, "F: \cong"]
&
E(\mathbb{C})
\end{tikzcd}
\]}
Therefore \(F \circ \pi\) and \(\exp_{E,\omega} \) are holomorphic group homomorphisms
\(\mathbb C\to E(\mathbb C)\) with the same pullback of $\omega$, so
\(F\circ \pi=\exp_{E,\omega}\).
Since \(\Ker\pi=\Lambda\) and \(\mathbb C/\Lambda\to E(\mathbb C)\) is
injective, we obtain $\Lambda(E,\omega)=\Ker(\exp_{E,\omega})=\Ker (F\circ \pi)=\Lambda$.\end{proof}

\begin{proposition}\label{commonperiod}

Let $K\subset\overline{\mathbb Q}\subset\mathbb C$ be a number field. Let \(E_1\) and \(E_2\) be elliptic curves over $K$.  For
\(i=1,2\), let \(\omega_i\) be a nonzero invariant differential on
\(E_i\) defined over $K$. Suppose that there exists a
nonzero complex number \(\Omega\) such that
\[
        \Omega \in \Lambda(E_1,\omega_1)
        \quad\text{and}\quad
        \Omega \in \Lambda(E_2,\omega_2),
\]
where \(\Lambda(E_i,\omega_i)\) denotes the period lattice of
\((E_i,\omega_i)\).  Then \(E_1\) and \(E_2\) are isogenous over $K$.
\end{proposition}

\begin{proof}
Put $G=E_1\times E_2$. Then, $\text{exp}_G=\text{exp}_{E_1}\times \text{exp}_{E_2}$. 
For each \(i=1,2\), let  $e_i\in \operatorname{Lie}(E_i)$ be the basis dual to \(\omega_i\), namely \(\omega_i(e_i)=1\).
Then
\[
        \operatorname{Lie}(G)_{\overline{\mathbb Q}}
        =
        \overline{\mathbb Q}e_1\oplus \overline{\mathbb Q}e_2.
\]
By Proposition \ref{delicate}, we have
\[
        \exp_{E_1}(\Omega e_1)=O_{E_1},
        \exp_{E_2}(\Omega e_2)=O_{E_2}.
\]
Hence, for $u:=(\Omega e_1,\Omega e_2)=\Omega \cdot (e_1,e_2)
        \in \operatorname{Lie}(G(\mathbb C))$, we have $\exp_G(u)=(O_{E_1},O_{E_2})\in G(\overline{\mathbb Q})$.

Let \(V_u\) be the smallest \(\overline{\mathbb Q}\)-vector subspace of
\(\operatorname{Lie}(G)\) such that $u\in V_u\otimes_{\overline{\mathbb Q}}\mathbb C$. Since \(u=\Omega \cdot (e_1,e_2)\) and \(\Omega\neq 0\), we have
\[
        V_u=\overline{\mathbb Q}\cdot (e_1,e_2).
\]
By W\"ustholz's analytic subgroup theorem (=Theorem \ref{annals}), there exists a connected
algebraic subgroup $H\subset E_1\times E_2$ such that
\[
        \operatorname{Lie}(H)=V_u
        =
        \overline{\mathbb Q}\cdot(e_1,e_2).
\]

We first show that \(H\) gives an isogeny over \(\overline{\mathbb Q}\).
Let $p_i:H\longrightarrow E_i$ be the restrictions of the two projections \(E_1\times E_2\to E_i\).
The differentials of these maps are the restrictions of the natural
projections $\operatorname{Lie}(E_1)\oplus \operatorname{Lie}(E_2)
        \longrightarrow \operatorname{Lie}(E_i)$.

 If $pr_1:H_u\to E_1$ were the zero map, then
\[
V_u=\operatorname{Lie} H_u \subset 0\times \operatorname{Lie} E_2.
\]
This contradicts the fact that the first component of $V_u$ is nonzero.

 Since \(H\) is a
one-dimensional connected algebraic subgroup of an abelian variety,
\(H\) is an elliptic curve.  Thus \(p_1\) and \(p_2\) are isogenies over
\(\overline{\mathbb Q}\).

It remains to prove that this isogeny descends to $K$.  Recall that
\[
        V_u=\overline{\mathbb Q}\cdot (e_1,e_2)
\]
is stable under \(G_K\), because \(e_1\) and \(e_2\) are
defined over $K$. Therefore, for every
\(\sigma\in \mathrm{Gal}(\overline{\Bbb{Q}}/K)\), we have
\[
        \operatorname{Lie}({}^{\sigma}H)
        =
        \sigma(\operatorname{Lie}(H))
        =
        \sigma(V_u)
        =
        V_u
        =
        \operatorname{Lie}(H).
\] 

Since \(E_1\times E_2\) is an abelian variety over a field of
characteristic zero, its connected algebraic subgroups are determined by
their Lie algebras (see [Proposition 13.11 of \cite{milne}]).

  Since $G$ is defined over $K$, we identify the ambient group
${}^{\sigma}G$ with $G$. Hence
\[
        {}^{\sigma}H=H
        \qquad
        \text{for all }\sigma\in \mathrm{Gal}(\overline{\mathbb{Q}}/K).
\]

Thus \(H\) is \(G_K\)-stable, and hence descends to an
algebraic subgroup of \(E_1\times E_2\) defined over $K$.

Consequently, the projections $p_i:H\longrightarrow E_i$ are isogenies defined over $K$. For instance, if
\(\widehat{p_1}:E_1\to H\) denotes the dual isogeny of \(p_1\), then
\[
        \varphi:=p_2\circ \widehat{p_1}:E_1\longrightarrow E_2
\]
is a nonzero homomorphism defined over $K$.  Therefore
\(E_1\) and \(E_2\) are isogenous over $K$.\end{proof}

\begin{theorem}\label{periodmain}

Let $K$ be a number field with a real embedding $\sigma\colon K\hookrightarrow\mathbb{R}$, and let $E_1$ and $E_2$ be elliptic curves
over $K$. For each $i=1,2$, fix a nonzero invariant differential
$\omega_{E_i}$ on $E_i$.

Let $E_{i,\sigma}$ denote the elliptic curve over $\mathbb{R}$ obtained
from $E_i$ by base change via $\sigma$, and let
$\omega_{E_{i,\sigma}}$ denote the invariant differential on
$E_{i,\sigma}$ induced by $\omega_{E_i}$.

Define
\[
\Omega_\sigma(E_i)
:=
\int_{E_{i,\sigma}(\mathbb{R})}
\left|\omega_{E_{i,\sigma}}\right|.
\]

Then
\[
\frac{\Omega_\sigma(E_1)}
     {\Omega_\sigma(E_2)}
\in \sigma(K^\times)
\]
if and only if $E_1$ and $E_2$ are isogenous over $K$.
\end{theorem}

\begin{proof}
Suppose first that $\frac{\Omega_\sigma(E_1)}
     {\Omega_\sigma(E_2)}
\in \sigma(K^\times)$.
Since this quotient is positive, there exists $a\in K^\times$ such that
$\sigma(a)>0$ and $\Omega_\sigma(E_1)
=
\sigma(a)\Omega_\sigma(E_2)$.
For each $i=1,2$, choose a generator $\gamma_i^+
\in
H_1(E_{i,\sigma}(\mathbb{C}),\mathbb{Z})^+$
such that $\int_{\gamma_i^+}\omega_{E_{i,\sigma}}>0$. Then
\[
\Omega_\sigma(E_i)
=
c_{i,\sigma}
\int_{\gamma_i^+}\omega_{E_{i,\sigma}},
\qquad
c_{i,\sigma}
=
[E_{i,\sigma}(\mathbb{R})
 :
 E_{i,\sigma}(\mathbb{R})^0]
\in\{1,2\}.
\]
Hence $\Omega_\sigma(E_i)
\in
\Lambda(E_{i,\sigma},\omega_{E_{i,\sigma}})$. Moreover,
\[
\Lambda
\bigl(
E_{2,\sigma},
\sigma(a)\omega_{E_{2,\sigma}}
\bigr)
=
\sigma(a)
\Lambda
\bigl(
E_{2,\sigma},
\omega_{E_{2,\sigma}}
\bigr).
\]
Therefore, the nonzero real number $\Omega_\sigma(E_1)
=
\sigma(a)\Omega_\sigma(E_2)$
belongs to $\Lambda
\bigl(
E_{1,\sigma},
\omega_{E_{1,\sigma}}
\bigr)
\cap
\Lambda
\bigl(
E_{2,\sigma},
\sigma(a)\omega_{E_{2,\sigma}}
\bigr)$. Thus, $E_{1,\sigma}$ and $E_{2,\sigma}$ have a nonzero common period
with respect to invariant differentials defined over $\sigma(K)$.
By Proposition~\ref{commonperiod}, they are isogenous over $\sigma(K)$.
Transporting this isogeny via
\[
\sigma^{-1}\colon \sigma(K)\longrightarrow K,
\]
we conclude that $E_1$ and $E_2$ are isogenous over $K$.

Conversely, suppose that $E_1$ and $E_2$ are isogenous over $K$, and let $\phi\colon E_1\longrightarrow E_2$
be a $K$-isogeny. Write $\phi^*\omega_{E_2}
=
a\omega_{E_1}$ 
for some $a\in K^\times$, and let $\phi_\sigma\colon
E_{1,\sigma}(\mathbb{R})
\longrightarrow
E_{2,\sigma}(\mathbb{R})$
be the induced homomorphism on real points.

By \cite[Lemma 7.4]{dokchitser},
\[
\frac{\Omega_\sigma(E_2)}
     {\Omega_\sigma(E_1)}
=
\frac{\#\operatorname{Coker}(\phi_\sigma)}
     {\#\operatorname{Ker}(\phi_\sigma)}
\left|\sigma(a)\right|
\in \sigma(K^\times).
\]
Hence
\[
\frac{\Omega_\sigma(E_1)}
     {\Omega_\sigma(E_2)}
\in \sigma(K^\times),
\]
which proves the converse.\end{proof}

\begin{example}
By Kenku's theorem \cite{Kenku}, each $\mathbb{Q}$-isogeny class
contains at most eight $\mathbb{Q}$-isomorphism classes of elliptic
curves. Hence, at most eight pairwise non-isomorphic elliptic curves
over $\mathbb{Q}$ can share the same real period.

The six curves
\texttt{98.a1}, \texttt{98.a2}, \texttt{98.a3},
\texttt{98.a4}, \texttt{98.a5}, and \texttt{98.a6}
are pairwise non-isomorphic over $\mathbb{Q}$ and have the same
real period. 
\end{example}

\begin{lemma}\label{lonely}
For every $\mathcal{E}>0$, the cardinality of 
\[
\{\Omega_E \mid E \ \text{is an elliptic curve defined over $\Bbb{Q}$}\}\cap[\mathcal{E},\infty)
\]
is finite. 
\end{lemma}
\begin{proof}
Let $H(E)$ denote the naive height of $E$ used in \cite{poonen}.
By Corollary 6.1.3 of \cite{poonen}, we have $\Omega_E \ll H(E)^{-1/12}\log H(E)$. Hence, if $\Omega_E \ge \mathcal{E}>0$, then $\mathcal{E} \ll H(E)^{-1/12}\log H(E)$.
Since $H^{-1/12}\log H \longrightarrow 0
\qquad (H\to\infty)$, it follows that $H(E)$ is bounded. Therefore, there are only finitely many such elliptic curves up to $\mathbb{Q}$-isomorphism.\end{proof}

\begin{theorem}\label{rounded}
Let $a \in \mathbb{Q}_{>0}$ and let $E_1/\mathbb{Q}$ be an elliptic curve.
Then there exists $\varepsilon(E_1,a)>0$ such that, for every elliptic curve
$E_2/\mathbb{Q}$, if
\[
\left|
\frac{\Omega_{E_1}}{\Omega_{E_2}}-a
\right|
<\varepsilon(E_1,a),
\]
then $\dfrac{\Omega_{E_1}}{\Omega_{E_2}}=a$. In particular, $E_1$ and $E_2$ are isogenous over $\mathbb{Q}$.
\end{theorem}

\begin{proof}

\[
\left|
\frac{\Omega_{E_1}}{\Omega_{E_2}}-a
\right|
<
\frac{a}{2}
\quad\Longrightarrow\quad
\frac{2\Omega_{E_1}}{3a}
<
\Omega_{E_2}
<
\frac{2\Omega_{E_1}}{a}.
\]

Thanks to Lemma \ref{lonely}, the set $T_a:=
\left\{
\frac{\Omega_{E_1}}{\Omega_E}
\;\middle|\;
E/\mathbb{Q}\text{ is an elliptic curve}, \frac{a}{2}
<
\frac{\Omega_{E_1}}{\Omega_E}
<
\frac{3a}{2}
\right\}$ is finite.

By choosing $\varepsilon$ small enough,
\[
\varepsilon(E_1,a)
:=
\min\left\{
\frac{a}{2},
\frac{\delta}{2}
\right\}, \text{where }
\delta :=\min\left(
\{a\}\cup
\left\{|x-a| \;\middle|\; x\in T_a\setminus\{a\}\right\}
\right).
\]

Then $T_a\cap(a-\varepsilon,a+\varepsilon)
\subseteq
\{a\}$. Thus, we have $\frac{\Omega_{E_1}}{\Omega_{E_2}}=a$.\end{proof}

\begin{remark}\label{counter}

In this remark, since $K$ is totally real, we have
\[
\Omega_{E/K}:=
\prod_{\sigma:K\hookrightarrow \mathbb{R}}\Omega_{\sigma}(E_{\sigma}),
\qquad
\Omega_{\sigma}(E):=
\int_{E_\sigma(\mathbb{R})}
\left|\omega_{E_{\sigma}}\right|.
\]

Here, $E_{\sigma}$ denotes the elliptic curve over $\Bbb{R}$ obtained from $E/K$ by base change via the embedding $\sigma:K\hookrightarrow \Bbb{R}$.

For a general number field, this statement does not hold. Indeed, there exist
examples of non-isogenous elliptic curves sharing the same real period.

Let $K=\mathbb{Q}(\phi)=\mathbb{Q}(\sqrt{5})$, where
$\phi=(1+\sqrt{5})/2$ is the golden ratio. Let
\[
E_1:\quad y^2=x^3+(\phi-1)x^2-2x-2\phi+4
\]
and
\[
E_2:\quad y^2=x^3-\phi x^2-2x+2\phi+2
\]
be elliptic curves over $K$. These models are global minimal models. Their LMFDB labels are $\texttt{2.2.5.1-4096.1-k3}
\ \text{and}\ 
\texttt{2.2.5.1-4096.1-n1}$,
respectively.

According to the LMFDB, these two curves belong to the distinct
isogeny classes \texttt{4096.1-k} and \texttt{4096.1-n}. Therefore, they are
not isogenous over $K$. To see this independently of the LMFDB labels, consider the prime $\mathfrak p=(29,\varphi-6)$
of \(K\). Since \(\varphi^2-\varphi-1\equiv 0 \bmod{29}\) for
\(\varphi\equiv 6\), the residue field is \(\mathbb F_{29}\).
Reducing the two curves modulo \(\mathfrak p\), one obtains $\#E_1(\mathbb F_{29})=36,
\#E_2(\mathbb F_{29})=24$.
Thus, $a_{\mathfrak p}(E_1)=29+1-36=-6,
a_{\mathfrak p}(E_2)=29+1-24=6$. If \(E_1\) and \(E_2\) were \(K\)-isogenous, then at every prime of good
reduction their reductions would have the same number of points over the
residue field. This contradiction shows that \(E_1\) and \(E_2\) are not
isogenous over \(K\).

On the other hand, they have the same real period. Let $\tau$ be the nontrivial
element of $\operatorname{Gal}(K/\mathbb{Q})$, so that $\tau(\phi)=1-\phi$.
Then
\[
E_1={}^{\tau}E_2.
\]
Let $\sigma_1,\sigma_2:K\hookrightarrow\mathbb{R}$ be the two real embeddings,
with $\sigma_2=\sigma_1\circ\tau$. Then
\[
\begin{aligned}
{}^{\sigma_1}E_1
&={}^{\sigma_1}\bigl({}^{\tau}E_2\bigr)\\
&={}^{\sigma_1\circ\tau}E_2\\
&={}^{\sigma_2}E_2,
\end{aligned}
\]
and similarly ${}^{\sigma_2}E_1={}^{\sigma_1}E_2$. Hence the action of $\tau$
interchanges the two local real periods. Therefore
\[
\begin{aligned}
\Omega_{E_1/K}
&=
\int_{{}^{\sigma_1}E_1(\mathbb{R})}
\left|{}^{\sigma_1}\omega_1\right|
\cdot
\int_{{}^{\sigma_2}E_1(\mathbb{R})}
\left|{}^{\sigma_2}\omega_1\right| \\
&=
\int_{{}^{\sigma_2}E_2(\mathbb{R})}
\left|{}^{\sigma_2}\omega_2\right|
\cdot
\int_{{}^{\sigma_1}E_2(\mathbb{R})}
\left|{}^{\sigma_1}\omega_2\right| \\
&=
\Omega_{E_2/K},
\end{aligned}
\]
where $\omega_i$ denotes the standard invariant differential on $E_i$. Thus
$E_1$ and $E_2$ give an example of non-isogenous elliptic curves over $K$ with
the same real period.

\end{remark}

The $L$-functions and the periods agree after interchanging places and embeddings.
However, if we define agreement of $L$-functions in the sense of
``when the local factors of their Euler products agree at every place,''
then they are different.
Likewise, for the periods, if we fix an embedding, then being isogenous over $K$
is forced, as in Theorem \ref{periodmain}.
As the following proposition shows, every BSD twin arises from an isogeny together with
conjugation.
However, if we interpret BSD twins in terms of the strong agreement of $L$-functions in the
sense of ``when the local factors of their Euler products agree at every place,''
or in terms of the strong agreement of periods in Theorem \ref{periodmain},
then the only BSD twins are those coming from an isogeny.

\begin{proposition}\label{possibility}
Let $K/\mathbb{Q}$ be a finite Galois extension, and let $A_1$ and
$A_2$ be simple abelian varieties over $K$. Then
\[
L(A_1/K,s)=L(A_2/K,s)
\]
if and only if
\[
A_1\sim_K {}^\tau A_2
\]
for some $\tau\in\operatorname{Gal}(K/\mathbb{Q})$.
\end{proposition}

\begin{proof}
Suppose first that $L(A_1/K,s)=L(A_2/K,s)$. Put $R_i:=\operatorname{Res}_{K/\mathbb{Q}}A_i (i=1,2)$. We have $L(R_i/\mathbb{Q},s)=L(A_i/K,s)$. Hence the assumption implies $L(R_1/\mathbb{Q},s)=L(R_2/\mathbb{Q},s)$.
By Faltings' isogeny theorem, we have $R_1\sim_{\mathbb{Q}}R_2$. After base change to $K$, we obtain a $K$-isogeny $(R_1)_K\sim_K(R_2)_K$. Since $K/\mathbb{Q}$ is Galois, there is an isomorphism of
$K$-algebras $K\otimes_{\mathbb{Q}}K \cong \prod_{\sigma\in\operatorname{Gal}(K/\mathbb{Q})}K$.
Therefore, by the base-change formula for Weil restriction and its
compatibility with finite products
\cite[2.59--2.61]{MilneAG}, we obtain a $K$-isomorphism $\left(\operatorname{Res}_{K/\mathbb{Q}}A_i\right)_K
\simeq_K
\prod_{\sigma\in\operatorname{Gal}(K/\mathbb{Q})}
{}^\sigma A_i$, where ${}^\sigma A_i
:=
A_i\times_{\operatorname{Spec}K,\,\sigma}\operatorname{Spec}K$ 
denotes the $\sigma$-conjugate of $A_i$. Thus, we have a $K$-isogeny
\[
\Phi:
\prod_{\sigma\in\operatorname{Gal}(K/\mathbb{Q})}
{}^\sigma A_1
\longrightarrow
\prod_{\sigma\in\operatorname{Gal}(K/\mathbb{Q})}
{}^\sigma A_2.
\]
Let $\iota_{\mathrm{id}}:
A_1\longrightarrow
\prod_{\sigma\in\operatorname{Gal}(K/\mathbb{Q})}
{}^\sigma A_1$
denote the inclusion into the factor corresponding to the identity
element. Since $\Phi$ has finite kernel, the composite
$\Phi\circ\iota_{\mathrm{id}}$ is nonzero. Therefore, there exists
some $\tau\in\operatorname{Gal}(K/\mathbb{Q})$ such that
\[
A_1
\xrightarrow{\ \iota_{\mathrm{id}}\ }
\prod_{\sigma\in\operatorname{Gal}(K/\mathbb{Q})}{}^\sigma A_1
\xrightarrow{\ \Phi\ }
\prod_{\sigma\in\operatorname{Gal}(K/\mathbb{Q})}{}^\sigma A_2
\xrightarrow{\ \operatorname{pr}_{\tau}\ }
{}^\tau A_2
\]
is nonzero. Since both $A_1$ and ${}^\tau A_2$ are simple abelian varieties over
$K$, every nonzero homomorphism between them is an isogeny. Hence $A_1\sim_K{}^\tau A_2$.

Conversely, suppose that $A_1\sim_K{}^\tau A_2$
for some $\tau\in\operatorname{Gal}(K/\mathbb{Q})$. Isogenous abelian
varieties have the same $L$-function, while Galois conjugation merely
permutes the places of $K$. Therefore, $L(A_1/K,s)
=
L({}^\tau A_2/K,s)
=
L(A_2/K,s)$.\end{proof}

\section{Recover elliptic curves over $\Bbb{Q}$}

Since the real period determines the isogeny class, if two elliptic curves have
the same real period, then their \(L\) functions automatically coincide. The
equality of real periods is a stronger condition than the equality of
\(L\) functions.

\begin{proposition}\label{rectangler}
Let \(E\) and \(E'\) be elliptic curves over \(\mathbb Q\). Suppose that $c_{E/\mathbb R}=c_{E'/\mathbb R}$ and \[\Omega_E^+=\Omega_{E'}^+,
\Omega_E^-=\Omega_{E'}^-.\] 
Then they can be uniformized by the same lattice. 

In other words, there exist a lattice $\Lambda$
and an isomorphism $\phi$ defined over $\Bbb{Q}$ such that the following diagram is commutative.

\[
\begin{tikzcd}
\mathbb{C}/\Lambda
  \arrow[r, "\text{id}: z \mapsto z", "\cong"']
  \arrow[d, "\cong"']
&
\mathbb{C}/\Lambda
  \arrow[d, "\cong"]
\\
E(\mathbb{C})
  \arrow[r, "\phi"', "\cong"]
&
E'(\mathbb{C})
\end{tikzcd}
\]

\end{proposition}

\begin{proof}

If \(E(\mathbb R)\) is disconnected, equivalently \(c_{E/\mathbb R}=2\), then the
lattice is rectangular and we may choose the generators \(r^+\) and \(r^-\) so
that $\Lambda(E,\omega_E)
=
\mathbb Z \Omega_E^+ \oplus \mathbb Z \Omega_E^-$.
Here \(\Omega_E^+\in \mathbb R_{>0}\) and \(\Omega_E^-\in i\mathbb R_{>0}\), up to
the choice of sign of \(r^-\).

If \(E(\mathbb R)\) is connected, equivalently \(c_{E/\mathbb R}=1\), then the
lattice is non-rectangular and we may choose the generators so that $\Lambda(E,\omega_E)
=
\mathbb Z \Omega_E^+
\oplus
\mathbb Z\left(\frac{\Omega_E^+ + \Omega_E^-}{2}\right)$.
 Since $c_{E/\mathbb R}=c_{E'/\mathbb R}$, the curves \(E\) and \(E'\) have the same lattice.

The common period lattice \(\Lambda\) determines \(c_4=12g_2(\Lambda)\) and \(c_6=216g_3(\Lambda)\), and hence the Weierstrass equation $y^2=x^3-\frac{c_4}{48}x-\frac{c_6}{864}$. Therefore, \(E\) and \(E'\) are \(\mathbb Q\)-isomorphic.\end{proof}

\begin{theorem}[\cite{pal}, Theorem 3.2, Corollary 2.6 ]
Let \(E/\mathbb{Q}\) be an elliptic curve given by a global minimal
Weierstrass equation, and let \(E^D\) denote its quadratic twist by a
square-free integer \(D\). Then the BSD real periods of \(E\) and \(E^D\)
satisfy the following formula.

If \(D>0\), then
\[
  \Omega_{E^D}=\frac{\widetilde{u}}{\sqrt{D}}\,\Omega_E.
\]

If \(D<0\), then, up to sign,
\[
  \Omega_{E^D}
  =
  \frac{\widetilde{u}}{\sqrt{D}}\,
  c_{E^D/\Bbb{R}}\,\Omega^{-}_E,
\] where $c_\infty(E^D)$ is the number of connected components of
$E^D(\mathbb{R})$ and $\tilde{u}$ is a rational number such that
$2\tilde{u} \in \mathbb{Z}$; it depends on $E$ and $D$, and is
defined explicitly in Proposition 2.5 of \cite{pal}(the elliptic curve $E$ in
Proposition 2.5 should be taken as a minimal model of the $E$ in this
theorem).
By Corollary 2.6 of \cite{pal}, when $D \equiv 1 \bmod 4$, $(D, \Delta_{E})=1$, we have $\tilde{u}=1$.

\end{theorem}

\begin{theorem}\label{recover}
Let $E/\mathbb{Q}$ be an elliptic curve, and let $D<0$ 
be a square-free integer satisfying
\[
D\equiv 1\bmod 4,\quad \gcd(D,\Delta_E)=1,
\]
where $\Delta_E$ denotes the minimal discriminant of $E$.
Then the following statements hold.

\begin{enumerate}
\item[(1)]
For any elliptic curve $E'/\mathbb{Q}$, if
\[
c_{E/\mathbb{R}}=c_{E'/\mathbb{R}},\quad
\Omega_E=\Omega_{E'},\quad
\Omega_{E^D}=\Omega_{E'^D},
\]
then $E$ and $E'$ are isomorphic over $\mathbb{Q}$.

\item[(2)]
There exists an integer $N=N(E,D)\geq 1$ such that,
for any elliptic curve $E'/\mathbb{Q}$, if
\[
c_{E/\mathbb{R}}=c_{E'/\mathbb{R}},\quad
R_N(\Omega_E)=R_N(\Omega_{E'}),\quad
R_N(\Omega_{E^D})=R_N(\Omega_{E'^D}),
\]
then $E$ and $E'$ are isomorphic over $\mathbb{Q}$.
\end{enumerate}

Here, $c_{E/\mathbb{R}}$ denotes the number of
connected components of $E(\mathbb{R})$, and
\[
R_N(x):=10^{-N}\left\lfloor 10^Nx+\frac12\right\rfloor
\quad (x>0).
\]
\end{theorem}

\begin{proof}
(1) Suppose that $\Omega_E=\Omega_{E'}$, $\Omega_{E^D}=\Omega_{E'^D}$, and $c_{E/\mathbb{R}}=c_{E'/\mathbb{R}}$. Since $\Omega_E=\Omega_{E'}$, we have \[
c_{E/\mathbb{R}}\Omega_E^+
=
c_{E'/\mathbb{R}}\Omega_{E'}^+.
\]

Let \(D\) be an integer satisfying \(D\equiv 1 \bmod 4\),
\((D,\Delta_E)=1\), and \(D<0\). Note that $c_{E^D/\mathbb{R}}=c_{E/\mathbb{R}}$ for any square-free integer $D$.

By Theorem \ref{periodmain}, $E$ and $E'$ are $\mathbb{Q}$-isogenous, and hence have the same primes of bad reduction. Therefore $E$ and $E'$ have the same set of primes of bad reduction. Hence, if $(D,\Delta_E)=1$, then $D$ is also coprime to $\Delta_{E'}$, and so $(D,\Delta_E\Delta_{E'})=1$.

Thus, taking one \(D\) such that \(D\equiv 1 \bmod 4\), $(D, \Delta_E)=1$  and \(D<0\), Corollary 2.6 of \cite{pal} gives $c_{E^D/\mathbb{R}}\Omega_E^-
=
c_{E'^D/\mathbb{R}}\Omega_{E'}^-$, that is, $c_{E/\mathbb{R}}\Omega_E^-
=
c_{E'/\mathbb{R}}\Omega_{E'}^-$. Hence $\Omega_E^+={\Omega_E'^+}, \Omega_{E}^- = \Omega_{E'}^-$.

Thus, by Proposition \ref{rectangler}, we can conclude that $E$ and $E'$ are isomorphic over $\Bbb{Q}$. 

(2) By Lemma \ref{lonely}, $\Omega_E$ is determined by its value rounded to its $n(E)$-th decimal place, and $\Omega_{E^D}$ is determined by its value rounded to its $n(E,D)$-th decimal place. Thus, by taking $N(E,D)\ge \text{max}({n(E),n(E,D)})$, we can say that $\Omega_E$ and $\Omega_{E^D}$ up to their $N(E,D)$-th decimal places determine the pair $(\Omega_E,\Omega_{E^D})$.\end{proof}

\begin{example}\label{twin2}[Distinguishing the BSD twin by negative quadratic twist]
Let \(m\) be either \(1\) or a prime satisfying $\left(\frac{390}{m}\right)=-1$,
and put \(n=m^2+64\). This construction provides a mechanism for producing many BSD twins \cite{shiga}.

Let \(E_m\) and \(E'_m\) be the elliptic curves
\[
E_m:\ y^2=x^3+390n x^2+195^2m^2n\,x,\qquad
E'_m:\ y^2=x^3-195n x^2+16\cdot 195^2n\,x.
\]

When \(m=1\), the pair \((E_1,E'_1)\) is a BSD twin with identical Kodaira symbols and minimal discriminants.

\FloatBarrier

{\small
\begin{table}[htbp]
\centering
\label{tab:elliptic_curves_38025_b}
\begin{tabular}{lll}
    \toprule
    \textbf{Elliptic curve} & $E_1: y^2=x^3+25350x^2+2471625x$ & $E'_1: y^2=x^3-50700x^2+632736000x$ \\
    \midrule
 j-invariant& $257^3$  & $17^3$  \\  
Minimal discriminant, $\Delta_E$ & $3^6\cdot 5^9\cdot 13^9 $ & $3^6\cdot 5^9\cdot 13^9 $ \\

    Mordell--Weil group & 
    $\mathbb{Z}/2\mathbb{Z}$ & $\mathbb{Z}/2\mathbb{Z}$ \\
    Regulator & 1 & 1 \\
  \rowcolor{gray!20}  Real period & $0.209\cdots$ & $0.209\cdots$ \\
  Tamagawa number & $2(p=3),2(p=5), 2(p=13)$ & $2(p=3), 2(p=5), 2(p=13)$\\

   Kodaira symbol  & $I_0^*(p=3), III^*(p=5), III^*(p=13)$,  & $I_0^*(p=3), III^*(p=5), III^*(p=13)$ \\

 $\Sha(E_i/\mathbb{Q})$ & same & same \\   
 
\rowcolor{gray!20}  $\Omega_{E^{-7}}$ & $0.029918\cdots$ & $0.059836\cdots$ \\

 \rowcolor{gray!20} $c_{E/\Bbb{R}}$ & 2 & 2 \\

    \bottomrule
\end{tabular}
\caption{The above pair cannot be distinguished by BSD invariants, Kodaira symbols, minimal discriminants. But they can be distinguished by the shaded invariants.}
\end{table}}

\FloatBarrier

In addition, since $\Delta_{E_1}=\Delta_{E'_1}>0$, we have $c_{E_1/\mathbb{R}}=c_{E'_1/\mathbb{R}}$. On the other hand, the real periods of the twists \(E_1^{-7}\) and \(E_1'^{-7}\) are distinct. Consequently, \(E_1/\mathbb{Q}\) is uniquely recovered from the data
\[
\left(\Omega_{E_1},c_{E_1/\Bbb{R}},\Omega_{E_1^{-7}}\right).
\]
\end{example}

\begin{remark}

Theorem \ref{recover} shows that an elliptic curve over $\mathbb{Q}$ can be recovered from data associated with three BSD invariants. However, no pair among these three invariants is sufficient in general.

Indeed, the pair $(\Omega_E,c_{E/\mathbb{R}})$ does not determine $E/\Bbb{Q}$, as shown by the pair $(E_1,E_1')$ in Example \ref{twin2}. Likewise, the pair $(\Omega_{E^D},c_{E/\mathbb{R}})$ does not determine $E$, as shown by $(E_1^{-11},(E_1')^{-11})$ and $D=-7$.

Finally, let $E:y^2=x^3+17x$
\text{and}
$E':y^2=x^3-68x$ be the pair in Example \ref{prop:j1728}. Then the pair $(\Omega_E,\Omega_{E^{-7}})$ does not determine $E/\Bbb{Q}$. Indeed, $\Omega_E=\Omega_{E'}
=
\frac{\Gamma(1/4)^2}{2\sqrt{\pi}\,17^{1/4}}$,
and $\Omega_{E^{-7}}=\Omega_{(E')^{-7}}
=
\frac{\Gamma(1/4)^2}{2\sqrt{7\pi}\,17^{1/4}}$. However, $c_{E/\Bbb{R}}=1, c_{E'/\Bbb{R}}=2$.

Thus, in general, an elliptic curve over $\mathbb{Q}$ can be recovered only from the full triple of invariants appearing in Theorem \ref{recover}, namely the two real periods and the Tamagawa number $c_{E/\mathbb{R}}$ at the archimedean place.

\end{remark}

\begin{theorem}\label{Cokerlemma}
Let $e_1<e_2<e_3$, and let
\[
E:\quad y^2=(x-e_1)(x-e_2)(x-e_3).
\]
For $i=1,2,3$, let
$\phi_i:E\to E/\langle(e_i,0)\rangle$
be the quotient isogeny. Then
\[
\operatorname{Coker}\bigl(\phi_i(E(\mathbb R))\bigr)\cong
\begin{cases}
\mathbb Z/2\mathbb Z & (i=1),\\
0 & (i=2,3).
\end{cases}
\]
\end{theorem}

\begin{proof}
This follows immediately from
\cite[Proposition 7.6]{dokchitser}
after the change of variables $x=X+e_i$.
\end{proof}

\begin{definition}
A claw graph is the graph consisting of one vertex of degree $3$
and three vertices of degree $1$, as illustrated below.
\[
\begin{tikzpicture}[
    every node/.style={circle, fill=black, inner sep=2pt}
]
    \node (v0) at (0,0) {};
    \node (v1) at (-1,-1) {};
    \node (v2) at (0,-1.5) {};
    \node (v3) at (1,-1) {};

    \draw (v0) -- (v1);
    \draw (v0) -- (v2);
    \draw (v0) -- (v3);
\end{tikzpicture}
\]
\end{definition}

\begin{lemma}\label{stevenslemma}[Stevens, 1989, Lemma 2.6 of \cite{stevens}]
Let $p$ be a prime, and let $A,B,C$ be elliptic curves
over $\mathbb{Q}_p$. Suppose that
\[
    f\colon A\longrightarrow C,
    \qquad
    g\colon B\longrightarrow C
\]
are \'etale isogenies of degree $p$.

Here, an isogeny $f\colon A\longrightarrow C$ is said
to be \'etale if
\[
    f^*\omega_C=\alpha\omega_A
    \qquad\text{for some }
    \alpha\in\mathbb{Z}_p^\times,
\]
where $\omega_A$ and $\omega_C$ are N\'eron differentials.

Then there exists an isomorphism
$\varphi\colon A\xrightarrow{\sim}B$ over $\mathbb{Q}_p$
such that the following diagram commutes:
\[
\begin{tikzcd}
A \arrow[rr,"\varphi","\sim"']
  \arrow[dr,"f"'] && B \arrow[dl,"g"] \\
& C &
\end{tikzcd}
\]
\end{lemma}

\begin{proof}
Consider the following commutative diagram:
\[
\begin{tikzcd}
& C \arrow[dl,"\widehat{f}"']
    \arrow[dr,"\widehat{g}"]
    \arrow[dd,"{[p]}",dashed] & \\
A \arrow[dr,"f"'] && B \arrow[dl,"g"] \\
& C &
\end{tikzcd}
\]

Suppose that there is no isomorphism $\varphi\colon A\xrightarrow{\sim}B$ satisfying $g\circ\varphi=f$.
Then $\mathrm{Ker}\widehat{f}\cap\mathrm{Ker}\widehat{g}=\{0\}$.
Consequently, the morphism $\widehat{f},\widehat{g})\colon
    C\longrightarrow A\times_C B$ is an isomorphism.

Let $\mathcal{N}_A$, $\mathcal{N}_B$, and
$\mathcal{N}_C$ denote the N\'eron models of
$A$, $B$, and $C$, respectively.
By the N\'eron mapping property, we obtain
\[
    \mathcal{N}_C
    \simeq
    \mathcal{N}_A\times_{\mathcal{N}_C}\mathcal{N}_B,
\]
since $\widetilde{f}$ and $\widetilde{g}$ are \'etale
and hence the fiber product is smooth over
$\mathbb{Z}_p$.

Thus we have the following Cartesian diagram:
\[
\begin{tikzcd}
& \mathcal{N}_C
    \arrow[dl,"\widetilde{\widehat{f}}"']
    \arrow[dr,"\widetilde{\widehat{g}}"]
    \arrow[dd,"\widetilde{[p]}",dashed] & \\
\mathcal{N}_A \arrow[dr,"\widetilde{f}"']
&&
\mathcal{N}_B \arrow[dl,"\widetilde{g}"] \\
& \mathcal{N}_C &
\end{tikzcd}
\]

Since \'etale morphisms are stable under base change,
$\widetilde{\widehat{f}}$ and
$\widetilde{\widehat{g}}$ are \'etale.
Therefore,
\[
    \widetilde{[p]}
    =
    \widetilde{f}\circ\widetilde{\widehat{f}}
\]
is \'etale, being a composition of \'etale morphisms.

On the other hand, since
\[
    [p]^*\omega_C=p\omega_C
    \qquad\text{and}\qquad
    p\notin\mathbb{Z}_p^\times,
\]
the morphism $\widetilde{[p]}$ is not \'etale.
This is a contradiction.
\end{proof}

\begin{lemma}[Claw graph does not share the same real period]
\label{claw}
Let $E$ and $E_i$ $(i=1,2,3)$ be pairwise non-isomorphic elliptic
curves over $\mathbb{Q}$. Suppose that, for each $i=1,2,3$,
there exists a $2$-isogeny
$\phi_i:E\to E_i$ defined over $\mathbb{Q}$.

Then
$\Omega_E=\Omega_{E_1}=\Omega_{E_2}=\Omega_{E_3}$
cannot occur.
\end{lemma}

\begin{proof}
We have
\[
1=\frac{\Omega_{E_i}}{\Omega_E}
=\frac{\#\operatorname{Coker}\phi_i}
{\#\operatorname{Ker}\phi_i}
\cdot
\left|
\frac{\phi_i^*\omega_{E_i}}{\omega_E}
\right|_{\mathbb{Q}_2}^{-1}.
\]
By Lemma \ref{Cokerlemma}, without loss of generality,
we can put
$\hat{\phi}_i^*\omega_E=\omega_{E_i}$
for $i=2,3$.

By Lemma \ref{stevenslemma},
$E_2$ and $E_3$ are isomorphic over $\mathbb{Q}$.
\end{proof}

\begin{remark}
A claw-type isogeny graph can share the same $\Omega_E^+$.
Indeed, $45.a1$, $45.a2$, $45.a3$, and $45.a5$
form a claw graph whose vertices share the same $\Omega_E^+$.

\begin{center}
\begin{tikzpicture}[
    >=Stealth,
    curve/.style={draw, rectangle, thick,
                  minimum width=1.6cm,
                  minimum height=0.65cm},
    every node/.style={font=\normalsize}
]


\node[curve] (E1) at (0,4) {$45.a1$};
\node[curve] (E2) at (0,2.3) {$45.a2$};
\node[curve] (E3) at (-2,0.5) {$45.a3$};
\node[curve] (E5) at (2,0.5) {$45.a5$};

\draw[thick] (E1) -- node[left] {$2$} (E2);
\draw[thick] (E2) -- node[above left] {$2$} (E3);
\draw[thick] (E2) -- node[above right] {$2$} (E5);


\end{tikzpicture}
\end{center}

\end{remark}

\begin{theorem}\label{atmost6}
The number of isomorphism classes of elliptic curves
over $\mathbb{Q}$ having the same real period
is at most $6$.
\end{theorem}

\begin{proof}
By Theorem \cite{Kenku}, the size of the isogeny graph
of an elliptic curve over $\mathbb{Q}$ is at most $8$.

By \cite{Robledo--Chiloyan}, an isogeny graph of size $8$
is achieved only by type $T_8$ or $S$.


\begin{figure}[htbp]
\centering
\begin{tikzpicture}[every node/.style={inner sep=2pt}]

\node (E1)  at (0,0)       {$E_1$};
\node (E2)  at (1.7,1.2)   {$E_2$};
\node (E21) at (1.7,2.7)   {$E_{21}$};
\node (E4)  at (3.4,0)     {$E_4$};
\node (E41) at (3.4,-1.45) {$E_{41}$};
\node (E8)  at (5.1,1.2)   {$E_8$};
\node (E81) at (5.1,2.7)   {$E_{81}$};
\node (E16) at (6.8,0)     {$E_{16}$};

\draw (E1) -- node[midway,above left] {$2$} (E2);
\draw (E2) -- node[midway,right] {$2$} (E21);
\draw (E2) -- node[midway,above right] {$2$} (E4);
\draw (E4) -- node[midway,right] {$2$} (E41);
\draw (E4) -- node[midway,above left] {$2$} (E8);
\draw (E8) -- node[midway,right] {$2$} (E81);
\draw (E8) -- node[midway,above right] {$2$} (E16);

\end{tikzpicture}
\caption{Isogeny graph of type $T_8$.}
\end{figure}


\begin{figure}[htbp]
\centering
\begin{tikzpicture}[every node/.style={inner sep=2pt}]

\node (E1)  at (0,2.5)    {$E_1$};
\node (E3)  at (2.1,2.5)  {$E_3$};
\node (E2)  at (0,1)      {$E_2$};
\node (E6)  at (2.1,1)    {$E_6$};
\node (E21) at (-1.45,-1) {$E_{21}$};
\node (E12) at (0.4,-1)   {$E_{12}$};
\node (E4)  at (1.7,-1)   {$E_4$};
\node (E31) at (3.55,-1)  {$E_{31}$};


\draw (E1) -- node[midway,above] {$3$} (E3);
\draw (E2) -- node[midway,above] {$3$} (E6);
\draw (E21) -- node[midway,above] {$3$} (E12);
\draw (E4) -- node[midway,above] {$3$} (E31);


\draw (E1) -- node[midway,left] {$2$} (E2);
\draw (E3) -- node[midway,right] {$2$} (E6);
\draw (E2) -- node[pos=.35,above left] {$2$} (E21);
\draw (E2) -- node[pos=.25,above right] {$2$} (E4);
\draw (E6) -- node[pos=.25,above left] {$2$} (E12);
\draw (E6) -- node[pos=.35,above right] {$2$} (E31);

\end{tikzpicture}
\caption{Isogeny graph of type $S$.}
\end{figure}

\FloatBarrier

Suppose that eight elliptic curves share the same real period.
Then the isogeny graph contains a claw graph, and thus,
by Lemma \ref{claw}, they cannot share the same real period.

Suppose that seven elliptic curves share the same real period.

\begin{center}
\begin{tikzpicture}[
    line cap=round,
    line join=round,
    edge/.style={line width=1pt},
    dotvertex/.style={circle,fill=black,inner sep=0pt,minimum size=5.5pt},
    circ/.style={line width=0.95pt},
    labelarrow/.style={-{Stealth[length=2.7mm,width=2mm]},line width=0.95pt}
]
  \coordinate (E)         at ( 0.00, 0.00);
  \coordinate (Eone)      at (-1.35, 0.85);
  \coordinate (EoneTop)   at (-1.35, 2.15);
  \coordinate (EoneLeft)  at (-2.68, 0.14);
  \coordinate (Etwo)      at ( 1.87, 0.61);
  \coordinate (EtwoTop)   at ( 1.90, 2.12);
  \coordinate (EtwoRight) at ( 3.10,-0.45);
  \coordinate (Ethree)    at ( 0.00,-1.50);

  \draw[edge] (E) -- (Eone) -- (EoneTop);
  \draw[edge] (Eone) -- (EoneLeft);
  \draw[edge] (E) -- (Etwo) -- (EtwoTop);
  \draw[edge] (Etwo) -- (EtwoRight);
  \draw[edge] (E) -- (Ethree);

  \foreach \P in {Eone,EoneTop,EoneLeft,Etwo,EtwoTop,EtwoRight,Ethree} {
    \draw[circ] (\P) circle[radius=0.43];
  }

  \foreach \P in {E,Eone,EoneTop,EoneLeft,Etwo,EtwoTop,EtwoRight,Ethree} {
    \node[dotvertex] at (\P) {};
  }

  \node[font=\Large] at (-3.63,1.48) {$E_1$};
  \draw[labelarrow] (-3.18,1.44)
    .. controls (-2.45,1.39) and (-1.75,1.03) .. (-1.59,0.98);

  \node[font=\Large] at (-0.02,1.49) {$E$};
  \draw[labelarrow] (0.00,1.19) -- (0.00,0.24);

  \node[font=\Large] at (3.48,1.45) {$E_2$};
  \draw[labelarrow] (3.31,1.19)
    .. controls (2.77,1.12) and (2.52,0.98) .. (2.27,0.80);

  \node[font=\Large] at (0.77,-3.07) {$E_3$};
  \draw[labelarrow] (0.56,-2.55) -- (0.20,-1.90);

\end{tikzpicture}
\end{center}

If no claw-type configuration is contained among the seven vertices,
the only remaining possibilities are the configurations indicated
by the circles in the isogeny graphs above.
However, in each of these cases, the remaining central vertex $E$
must also have the same real period.
Indeed, we have
\[
\Omega_{E_i}/\Omega_E\in\{1/2,1\}
\quad \text{for } i=2,3,
\]
and
\[
\Omega_{E_1}/\Omega_E\in\{1,2\}.
\]
Since $\Omega_{E_1}=\Omega_{E_2}=\Omega_{E_3}=\Omega$,
it follows that $\Omega_E=\Omega$.
Consequently, all eight vertices in the isogeny graph
would share the same real period,
which is impossible by Lemma \ref{claw}.\end{proof}

\begin{example}[BSD sextuplets]
The following six pairwise non-isomorphic elliptic curves over
$\mathbb{Q}$ share the same BSD invariants:
\[
\left\{
\texttt{7056.bd1},\;
\texttt{7056.bd2},\;
\texttt{7056.bd3},\;
\texttt{7056.bd4},\;
\texttt{7056.bd5},\;
\texttt{7056.bd6}
\right\}.
\]

\end{example}

\section{Acknowledgements}
I would like to express my sincere gratitude to Prof. Nobuo Tsuzuki for encouraging me to pursue the question studied in this paper, and for his constructive advice and comments in seminars. I am also deeply grateful to Prof. Christian Wüthrich for helpful discussions during my stay in Nottingham, especially on plus and minus periods, transcendence. I would also like to thank Rikuto Ito, Yoshinori Kanamura for their helpful comments.

\end{document}